\documentclass{article}
\usepackage{amsmath,amssymb,mathtools,amsthm}
\usepackage{booktabs}
\usepackage{graphicx,xcolor}
\usepackage{tabularx}
\DeclarePairedDelimiter\floor{\lfloor}{\rfloor}
\newtheorem{mytheorem}{Theorem}
\newcommand{\Mod}[1]{\ (\text{mod}\ #1)}

\theoremstyle{definition}

\begin{document}

\title{Combinatorial Formula for The Partition Function}
\author{Zhumagali Shomanov}
\date{\today}
\maketitle

\begin{abstract}

In this article we will derive a combinatorial formula for the partition function p(n). In the second part of the paper we will establish connection between partitions and q-binomial coefficients and give new interpretation for q-binomial coefficients. 

\vspace{3mm}


\vspace{3mm}


\end{abstract}

\section{Introduction}

In this paper I will derive a combinatorial formula for the partition function. I discovered this formula by investigating a tree of the partition function. It appears that the set of all partitions is not homogeneous and different partitions belong to different levels.
 
The formula that will be discussed in this paper has also been discovered by Andrew Sills\cite{Sills} from Georgia Southern University. In his paper Professor Sills uses Durfee Squares to prove the formula. In my work, as I have mentioned above, I derive the formula from the tree structure of the partition function. 

I learned about work of Professor Sills accidentally, after my colleague Francesco Sica had written a letter to George Andrews for the quick assessment of the formula. Professor Andrews, in turn, forwarded this letter to Andrew Sills, who sent me his own article and suggested me to think about open questions which he posted in his article. I believe that my approach is different from his and uses another combinatorial methods.

In the second part of the paper I establish connection between partitions and q-binomial coefficients and give new combinatorial meaning for the q-binomial coefficients.

\section{Formulas for the partition function}

Partition function $p(n)$ is the number of ways of writing integer $n$ as a sum of positive integers, where the order of addends is not important. For example, since

\begin{displaymath}
4=1+1+1+1=1+1+2=1+3=2+2
\end{displaymath}
$p(4)=5$. By convention $p(n)=1$ when $n=0$ and $p(n)=0$ when $n<0$.

It is very hard to compute $p(n)$ for big $n$ by simply writing down all partitions of $n$. The table below shows values of $p(n)$ for different $n$'s.

\vspace{10mm}

\begin{tabular}{cccc}
\toprule

  n    	& p(n)            	& n      	& p(n) 						   	\\ 
\midrule
 10   	& 42               	& 200  	& 3972999029388 				   		\\
 50   	& 204226       	& 500   	& 2300165032574323995027 		   		\\
100 	& 190569292 	& 1000 	& 24061467864032622473692149727991 		\\
\bottomrule
\end{tabular}

\vspace{10mm}

Therefore, one may ask whether there are some formulas for the partition function?

One of the first formulas was obtained by L. Euler:

\begin{multline*}
p(n)=p(n-1)+p(n-2)-p(n-5)-p(n-7)+\ldots+\\
+(-1)^{k-1}p\left(n-\cfrac{k(3k-1)}{2}\right)+(-1)^{k-1}p\left(n-\cfrac{k(3k+1)}{2}\right)+\ldots.	
\end{multline*}
This recurrence relation was used by Major MacMahon to compute values of $p(n)$ for $n=0,\ldots,200$. However, its drawback is that to compute $p(n)$ for some $n$ one needs to know previous values of $p(n)$.

One of the crowning achievements in the theory of partitions is the exact formula for $p(n)$ undertaken and mostly completed by G. H. Hardy and S. Ramanujan\cite{Ramanujan} and fully completed and perfected by H. Rademacher\cite{Rademacher}:

\begin{displaymath}
p(n)=\cfrac1{\pi\sqrt{2}}\sum_{k=1}^\infty \sqrt{k}A_k(n)\cfrac{\mathrm{d}}{\mathrm{d}n}\cfrac{\sinh\left(\cfrac{\pi\lambda_n}{k}\sqrt{\cfrac23}\right)}{\lambda_n},
\end{displaymath}	
where $\lambda_n=\sqrt{n-\cfrac1{24}}$, $\displaystyle A_k(n)=\sum_{\substack{h(\mathrm{mod}\; k) \\ (h,k)=1}}\omega_{h,k}\exp^{-2\pi ihn/k}$ and $\omega_{h,k}$ is equal to

\[ 
\begin{cases}
	\left(\cfrac{-k}{h}\right)\exp\left(-\pi i\left(\cfrac14(2-hk-h)+\cfrac1{12}(k-k^{-1})(2h-h'+h^2h')\right)\right) & h\;\text{odd} \\
	\left(\cfrac{-h}{k}\right)\exp\left(-\pi i\left(\cfrac14(k-1)+\cfrac1{12}(k-k^{-1})(2h-h'+h^2h')\right)\right) & k\;\text{odd}
\end{cases}
\]
Here $h'$ is a solution to the congruence $hh'\equiv-1\,(\mathrm{mod}\,k)$.

In 2011 Ken Ono and Jan Hendrik Bruinier\cite{OnoBruinier} obtained an algebraic formula for the partition function in the following way. First, they define weight -2 weakly holomorphic modular form

\begin{displaymath}
F(z)=\cfrac12\cdot\cfrac{E_2(z)-2E_2(2z)-3E_2(3z)+6E_2(6z)}{\eta^2(z)\eta^2(2z)\eta^2(3z)\eta^2(6z)}
\end{displaymath}

Here $E_2(z)=1-24\sum_{n=1}^\infty\sigma(n)q^n$, $\eta(z)=q^{1/24}\prod_{n=1}^\infty(1-q^n)$ and $q=e^{2\pi iz}$. Then they define

\begin{displaymath}
P(z)=-\left(\cfrac{1}{2\pi i}\,\cfrac{d}{dz}+\cfrac1{2\pi y}\right)F(z)
\end{displaymath}

In the next step they consider the set of primitive integer binary quadratic forms $Q(x,y)=ax^2+bxy+cy^2$ with discriminant $1-24n$ and such that $6\,|\, a$ and $b\equiv1\Mod{12}$. The set of equivalence classes of such forms under the action of $\Gamma_0(6)$ is denoted by $\mathcal{Q}_n$. Then

\begin{displaymath}
p(n)=\cfrac1{24n-1}\sum_{Q\in\mathcal{Q}_n} P(\alpha_Q)
\end{displaymath}

where $\alpha_Q$ is the root of $Q(x,1)=0$ that lies in upper half of the complex plane.

Above we have seen analytic and algebraic formulas for the partition function. The next section is devoted to derivation of a combinatorial formula.

\section{Combinatorial formula for $p(n)$}

Notation:

\begin{align*}
&S_1(n)=n \\
&S_2(n)=\sum_{i=1}^{\floor*{\frac{n-2}{2}}} i(n-3-2(i-1)) \\
&S_{3j}(n)=j\sum_{i=1}^{\floor*{\frac{n-7-3(j-1)}{2}}} i(n-8-2(i-1)-3(j-1)) \\
&S_{4kj}(n)=kj\sum_{i=1}^{\floor*{\frac{n-14-3(j-1)-4(k-1)}{2}}} i(n-15-2(i-1)-3(j-1)-4(k-1)) \\ 
&\ldots \\
&S_{\floor*{\sqrt{n}}zyx\ldots kj}=zyx\ldots kj\sum_{i=1}^{\floor*{\frac{n-(\floor*{\sqrt{n}}^2-2)-3(j-1)-\ldots-\floor*{\sqrt{n}}(z-1)}{2}}} i(n-(\floor*{\sqrt{n}}^2-1)- \\
&-2(i-1)-3(j-1)-4(k-1)-\ldots-\floor*{\sqrt{n}}(z-1)) \\ 
\end{align*}

Then:

\begin{align*}
&S_3(n)=\sum_{j=1}^{\floor*{\frac{n-6}{3}}} S_{3j}(n) \\
&S_4(n)=\sum_{k=1}^{\floor*{\frac{n-12}{4}}}\sum_{j=1}^{\floor*{\frac{n-13-4(k-1)}{3}}} S_{4kj}(n) \\
&\ldots \\
&S_{\floor*{\sqrt{n}}}=\sum_{z=1}^{\floor*{\frac{n-({\floor*{\sqrt{n}}}^2-{\floor*{\sqrt{n}}})}{{\floor*{\sqrt{n}}}}}}\sum_{y=1}^{\floor*{\frac{n-({\floor*{\sqrt{n}}}^2-({\floor*{\sqrt{n}}}-1))-\floor*{\sqrt{n}}(z-1)}{\floor*{\sqrt{n}}-1}}}\ldots\sum_{j=1}^{\floor*{\frac{n-({\floor*{\sqrt{n}}}^2-3)-\floor*{\sqrt{n}}(z-1)-\ldots-4(k-1)}{3}}} S_{\floor*{\sqrt{n}}zyx\ldots kj}(n)
\end{align*}

\begin{mytheorem}
\begin{displaymath}
p(n)=S_1(n)+S_2(n)+S_3(n)+\ldots+S_{\floor*{\sqrt{n}}}(n)
\end{displaymath}
\end{mytheorem}

\begin{proof}

We will try to construct the tree of the partition function, from which the formula will follow immediately. The tree will follow the following logic: in each node the number of outputs is one more than the number of inputs. Below is the picture of the tree for $n=0,\ldots,3$.

\begin{center}
\includegraphics[scale=.35]{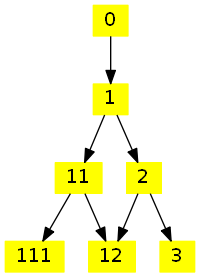}
\end{center}

Partitions in yellow form $S_1(n)$. Since node $12$ has two inputs, there has to be three outputs. Two of them are $112$ and $13$, and the third one is completely new, $22$. Partitions growing out of 22 form the first summand of $S_2(n)$.

\begin{center}
\includegraphics[scale=.35]{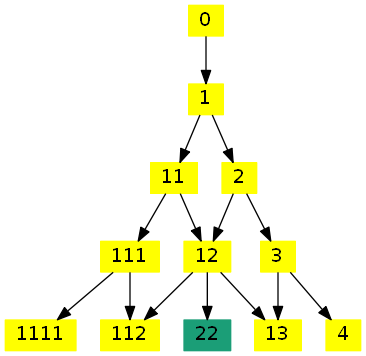}
\end{center}

Next picture shows partition tree for $n=0,\ldots,5$

\begin{center}
\includegraphics[scale=.35]{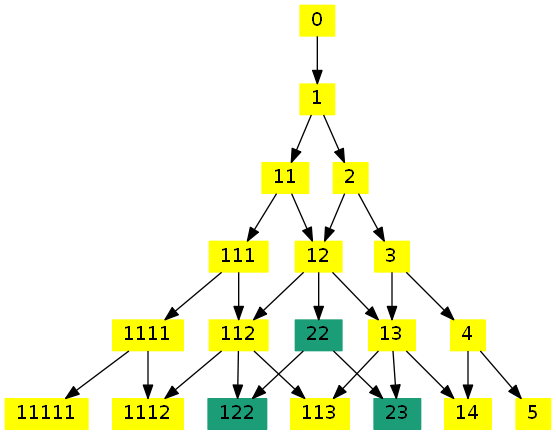}
\end{center}

Each of partitions 122 and 23 have two inputs, and therefore they should have three outputs. Outputs of 122 are 1122, 222, 123, and the outputs of 23 are 123, 33, 24. Partitions growing out of 222 and 33 together form the second summand of $S_2(n)$.

\begin{center}
\includegraphics[scale=.35]{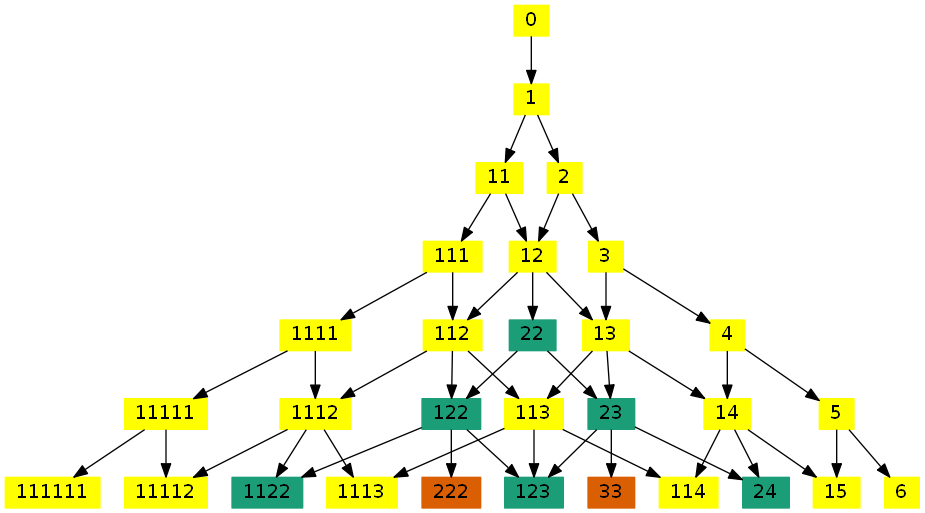}
\end{center}

The next two pictures show evolution of the tree for $n=7$ and $n=8$ (Yellow partitions are omitted in the interest of readability).

\begin{center}
\includegraphics[scale=.35]{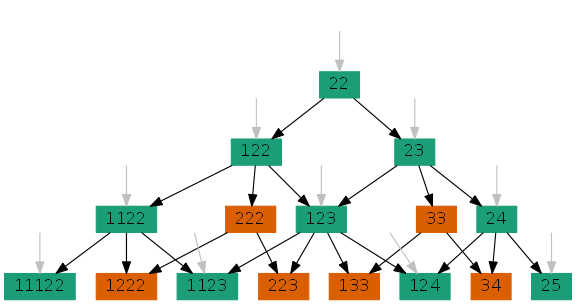}
\end{center}

\begin{center}
\includegraphics[scale=.44]{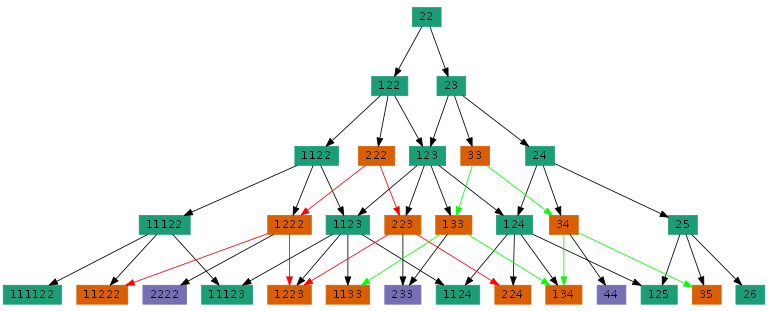}
\end{center}

Partitions growing out of 2222, 233 and 44 (violet partitions in the picture above) correspond to the term $3(n-7)$ in the formula for $S_2(n)$. Moreover, partition 233 gives rise to one completely new partition --- 333. This partition does not belong to $S_2(n)$ and starts the branch belonging to $S_{31}(n)$.

\begin{center}
\includegraphics[scale=.5]{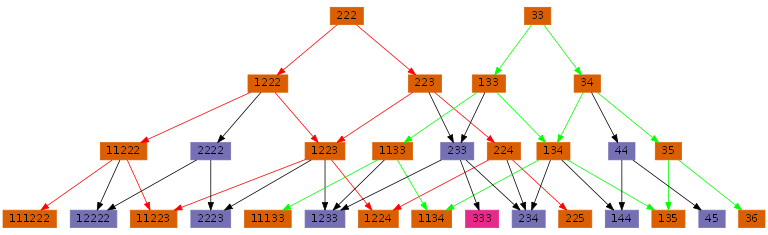}
\end{center}

\begin{center}
\includegraphics[scale=.5]{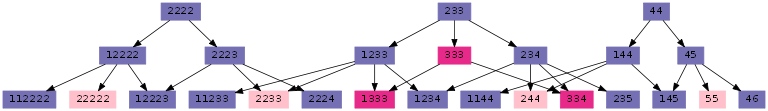}
\end{center}

Let us consider the branch starting from 333 separately:

\begin{center}
\includegraphics[scale=.4]{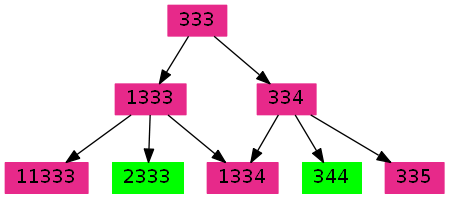}
\end{center}

\begin{center}
\includegraphics[scale=.4]{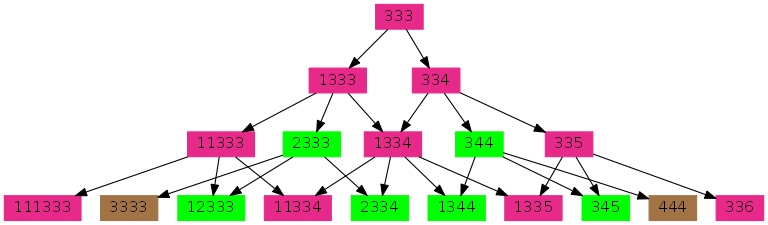}
\end{center}

One may notice the similarity in the structures of partitions growing out of 22 and partitions growing out of 333. However, partitions 3333 and 444 break this similarity. Therefore, it would be reasonable to consider them separately. It appears that each of these two branches behaves exactly as the branch starting from 333. They account for $S_{32}(n)$. 

\includegraphics[scale=.43]{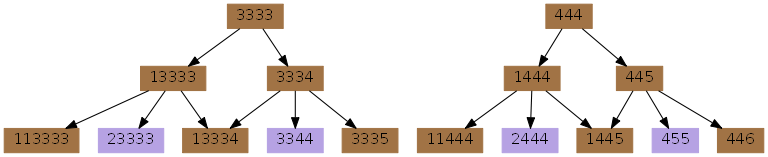}

When $n=15$ there will be three new branches that form $S_{33}(n)$, when $n=18$ there will be four branches that form $S_{34}(n)$ and so on. Therefore, each branch of $S_3(n)$ repeats the structure of $S_2(n)$. When $n=16$ we get one partition that does not fit into $S_1(n)$, $S_2(n)$ or $S_3(n)$ --- 4444. This partition starts a branch corresponding to $S_{41}(n)=\sum_{j=1}^{\floor*{(n-13)/3}} S_{41j}(n)$. This branch repeats the structure of a branch corresponding to $S_3(n)$. When $n=20$ we get two new partitions 44444 and 5555 that start branches that are counted by $S_{42}(n)=\sum_{j=1}^{\floor*{(n-17)/3}} S_{42j}$. Each of them repeats the structure of a branch corresponding to $S_3(n)$. On $n=24$ we get three such branches (444444, 45555, 6666), on $n=28$ --- four (4444444, 445555, 46666, 7777) and so on. In general, branches of the next level repeat the structure of the previous one. This is why as $k$ increases the formula becomes more and more involved.

\end{proof}

For instance, let us compute $p(21)$. Since $\floor*{\sqrt{21}}=4$, $p(21)=21+S_2(21)+S_3(21)+S_4(21)$.

\begin{align*}
&S_2(21)=\sum i(21-3-2(i-1))=\sum i(18-2(i-1))=18+2\cdot16+3\cdot14+\\
&+\ldots+8\cdot4+9\cdot2=330 \\
&S_{31}(21)=\sum i(21-8-2(i-1))=\sum i(13-2(i-1))=13+2\cdot11+3\cdot9+\\
&+\ldots+6\cdot3+7=140 \\
&S_{32}(21)=2\sum i(21-8-2(i-1)-3)=2\sum i(10-2(i-1))=2(10+2\cdot8+\\
&+3\cdot6+4\cdot4+5\cdot2)=140\\
&S_{33}(21)=3\sum i(21-8-2(i-1)-6)=3\sum i(7-2(i-1))=3(7+2\cdot5+\\
&+3\cdot3+4)=90\\
&S_{34}(21)=4\sum i(21-8-2(i-1)-9)=4\sum i(4-2(i-1))=4(4+2\cdot2)=32\\
&S_{35}(21)=5\sum i(21-8-2(i-1)-12)=5\sum i(1-2(i-1))=5\\
&S_{411}(21)=\sum i(21-15-2(i-1))=\sum i(6-2(i-1))=6+2\cdot4+3\cdot2=20\\
&S_{412}(21)=2\sum i(21-15-2(i-1)-3)=2\sum i(3-2(i-1))=2(3+2)=10\\
&S_{421}(21)=2\sum i(21-15-2(i-1)-4)=2\sum i(2-2(i-1))=2\cdot2=4
\end{align*}

Therefore, 
\begin{align*}
&S_3(21)=S_{31}+S_{32}+S_{33}+S_{34}+S_{35}=140+140+90+32+5=407 \\
&S_4(21)=S_{411}+S_{412}+S_{421}=20+10+4=34
\end{align*}
and
\begin{displaymath}
p(21)=21+330+407+34=792
\end{displaymath}

The table below shows the first few values of $S_k(n)$.

\vspace{10mm}

\begin{tabular}{ccccccccccccccccc}

\toprule
  n    		& 1      & 2      & 3 & 4 & 5 & 6 & 7 & 8 & 9 & 10 & 11&12&13&14&15&16					   	\\ 
\midrule
 $S_1(n)$   	& 1      & 2  	& 3 & 4 & 5 & 6 & 7 & 8 & 9 & 10	& 11&12&13&14&15&16 		   		\\
 $S_2(n)$   	& 	 &    	&    & 1 & 2 & 5 & 8 & 14 & 20 & 30 & 40&55&70&91&112&140		   		\\
$S_{31}(n)$ 	& 	 &  	&    &    &    &    &    &      & 1  & 2 & 5&8&14&20&30&40		\\
$S_{32}(n)$ & & & & & & & & & & & & 2 & 4 & 10 & 16 & 28 \\
$S_{33}(n)$ & & & & & & & & & & & & & & &3&6 \\
$S_{411}(n)$ & & & & & & & & & & & & & & & &1\\
\bottomrule
\end{tabular}

\vspace{10mm}

\section{Closed form expressions for $S_k(n)$}

If you try to sum $S_2(n)$ for all $n$, you will get a formula that involves floor function. However, if we consider only odd or even $n$ then we get:

\begin{align*}
S_2(n|n\;\text{odd})=2\begin{pmatrix} \cfrac{n-5}{2} \\[10pt] 0 \end{pmatrix}+6\begin{pmatrix} \cfrac{n-5}{2} \\[10pt] 1 \end{pmatrix}+6\begin{pmatrix} \cfrac{n-5}{2} \\[10pt] 2 \end{pmatrix}+2\begin{pmatrix} \cfrac{n-5}{2} \\[10pt] 3 \end{pmatrix} \\
S_2(n|n\;\text{even})=\begin{pmatrix} \cfrac{n-4}{2} \\[10pt] 0 \end{pmatrix}+4\begin{pmatrix} \cfrac{n-4}{2} \\[10pt] 1 \end{pmatrix}+5\begin{pmatrix} \cfrac{n-4}{2} \\[10pt] 2 \end{pmatrix}+2\begin{pmatrix} \cfrac{n-4}{2} \\[10pt] 3 \end{pmatrix}
\end{align*} 

In case of $S_3(n)$ we have to consider six different cases:

\begin{align*}
&S_3(n\,|\,n\equiv0\Mod{3}, n\;\text{odd})=\begin{pmatrix} \cfrac{n-9}{6} \\[10pt] 0 \end{pmatrix}+48\begin{pmatrix} \cfrac{n-9}{6} \\[10pt] 1 \end{pmatrix}+310\begin{pmatrix} \cfrac{n-9}{6} \\[10pt] 2 \end{pmatrix}+\\
&+695\begin{pmatrix} \cfrac{n-9}{6} \\[10pt] 3 \end{pmatrix}+648\begin{pmatrix} \cfrac{n-9}{6} \\[10pt] 4 \end{pmatrix}+216\begin{pmatrix} \cfrac{n-9}{6} \\[10pt] 5 \end{pmatrix}\\
&S_3(n\,|\,n\equiv1\Mod{3}, n\;\text{even})=2\begin{pmatrix} \cfrac{n-10}{6} \\[10pt] 0 \end{pmatrix}+72\begin{pmatrix} \cfrac{n-10}{6} \\[10pt] 1 \end{pmatrix}+390\begin{pmatrix} \cfrac{n-10}{6} \\[10pt] 2 \end{pmatrix}+\\
&+788\begin{pmatrix} \cfrac{n-10}{6} \\[10pt] 3 \end{pmatrix}+684\begin{pmatrix} \cfrac{n-10}{6} \\[10pt] 4 \end{pmatrix}+216\begin{pmatrix} \cfrac{n-10}{6} \\[10pt] 5 \end{pmatrix}\\
&S_3(n\,|\,n\equiv2\Mod{3}, n\;\text{odd})=5\begin{pmatrix} \cfrac{n-11}{6} \\[10pt] 0 \end{pmatrix}+105\begin{pmatrix} \cfrac{n-11}{6} \\[10pt] 1 \end{pmatrix}+483\begin{pmatrix} \cfrac{n-11}{6} \\[10pt] 2 \end{pmatrix}+\\
&+887\begin{pmatrix} \cfrac{n-11}{6} \\[10pt] 3 \end{pmatrix}+720\begin{pmatrix} \cfrac{n-11}{6} \\[10pt] 4 \end{pmatrix}+216\begin{pmatrix} \cfrac{n-11}{6} \\[10pt] 5 \end{pmatrix}\\
&S_3(n\,|\,n\equiv0\Mod{3}, n\;\text{even})=10\begin{pmatrix} \cfrac{n-12}{6} \\[10pt] 0 \end{pmatrix}+148\begin{pmatrix} \cfrac{n-12}{6} \\[10pt] 1 \end{pmatrix}+590\begin{pmatrix} \cfrac{n-12}{6} \\[10pt] 2 \end{pmatrix}+\\
&+992\begin{pmatrix} \cfrac{n-12}{6} \\[10pt] 3 \end{pmatrix}+756\begin{pmatrix} \cfrac{n-12}{6} \\[10pt] 4 \end{pmatrix}+216\begin{pmatrix} \cfrac{n-12}{6} \\[10pt] 5 \end{pmatrix}\\
&S_3(n\,|\,n\equiv1\Mod{3}, n\;\text{odd})=18\begin{pmatrix} \cfrac{n-13}{6} \\[10pt] 0 \end{pmatrix}+203\begin{pmatrix} \cfrac{n-13}{6} \\[10pt] 1 \end{pmatrix}+712\begin{pmatrix} \cfrac{n-13}{6} \\[10pt] 2 \end{pmatrix}+\\
&+1103\begin{pmatrix} \cfrac{n-13}{6} \\[10pt] 3 \end{pmatrix}+792\begin{pmatrix} \cfrac{n-13}{6} \\[10pt] 4 \end{pmatrix}+216\begin{pmatrix} \cfrac{n-13}{6} \\[10pt] 5 \end{pmatrix}\\
\end{align*}

\begin{align*}
&S_3(n\,|\,n\equiv2\Mod{3}, n\;\text{even})=30\begin{pmatrix} \cfrac{n-14}{6} \\[10pt] 0 \end{pmatrix}+272\begin{pmatrix} \cfrac{n-14}{6} \\[10pt] 1 \end{pmatrix}+850\begin{pmatrix} \cfrac{n-14}{6} \\[10pt] 2 \end{pmatrix}+\\
&+1220\begin{pmatrix} \cfrac{n-14}{6} \\[10pt] 3 \end{pmatrix}+828\begin{pmatrix} \cfrac{n-14}{6} \\[10pt] 4 \end{pmatrix}+216\begin{pmatrix} \cfrac{n-14}{6} \\[10pt] 5 \end{pmatrix}
\end{align*}

In the case of $S_4(n)$ there are already 12 formulas and in the case of $S_5(n)$ there are 60 formulas.

\section{Q-binomial coefficients}

We define q-binomial coefficient $\begin{bmatrix} n \\ k \end{bmatrix}$ as:

\begin{displaymath}
\begin{bmatrix} n \\ k \end{bmatrix}=\begin{cases} \cfrac{(1-q^n)(1-q^{n-1})\cdots(1-q^{n-k+1})}{(1-q)(1-q^2)\cdots(1-q^k)} & \mbox{if } k\le n \\ 0 & \mbox{if } k>n \end{cases}.
\end{displaymath}

At first glance it may appear that $\begin{bmatrix} n \\ k \end{bmatrix}$ is a rational function, but it can be proven that it is always a polynomial. Let

\begin{displaymath}
[m]_q=\cfrac{1-q^m}{1-q}
\end{displaymath}

Then

\begin{displaymath}
\begin{bmatrix} n \\ k \end{bmatrix}=\cfrac{[n]_q[n-1]_q\cdots[n-k+1]_q}{[1]_q[2]_q\cdots[k]_q}
\end{displaymath}

or if we denote $[m]_q!=[1]_q[2]_q\cdots[m]_q$

\begin{displaymath}
\begin{bmatrix} n \\ k \end{bmatrix}=\cfrac{[n]_q!}{[k]_q![n-k]_q!}
\end{displaymath}

When $q$ approaches one $\begin{bmatrix} n \\ k \end{bmatrix}$ transforms into usual binomial coefficient. Therefore, it would be normal to expect that some facts about binomial coefficients have their analogs in case of q-binomial coefficients. And indeed it is so. For example, there are analogs of Pascal identities:

\begin{displaymath}
\begin{bmatrix} n \\ k \end{bmatrix}=q^r\begin{bmatrix} n-1 \\ k \end{bmatrix}+\begin{bmatrix} n-1 \\ k-1 \end{bmatrix}
\end{displaymath} 

and

\begin{displaymath}
\begin{bmatrix} n \\ k \end{bmatrix}=\begin{bmatrix} n-1 \\ k \end{bmatrix}+q^{n-k}\begin{bmatrix} n-1 \\ k-1 \end{bmatrix}
\end{displaymath}

and q-analog of the binomial theorem

\begin{displaymath}
\prod_{k=0}^{n-1}(1+q^kt)=\sum_{k=0}^nq^{k(k-1)/2}\begin{bmatrix} n \\ k \end{bmatrix}t^k
\end{displaymath}

By doing some manipulations and taking limit as $n$ approaches to infinity one can derive the infinite $q$-binomial theorem:

\begin{displaymath}
\sum_{k=0}^\infty\cfrac{(1-a)(1-qa)\cdots(1-q^{k-1}a)}{(1-q)(1-q^2)\cdots(1-q^k)}t^k=\prod_{k=1}^\infty\cfrac{(1-q^kat)}{(1-q^kt)}
\end{displaymath}

This theorem is very important in the theory of basic hypergeometric series. It is also possible to derive from the finite $q$-binomial theorem the Jacobi triple product identity:

\begin{displaymath}
\prod_{j=1}^\infty(1+q^{2j-1}z)(1+q^{2j-1}z^{-1})(1-q^{2j})=\sum_{j=-\infty}^\infty q^{j^2}z^j
\end{displaymath}

Now, consider the following grids:

$$
\begin{array}{cccccccccccccccccc}
\vspace{5mm}
 1 \\
 1 \\
\vspace{5mm}
 1 \\
 1 \\
 1 & \textcolor{red}{1} \\
\vspace{5mm}
 1 \\
 \textcolor{red}{1} \\
 \textcolor{red}{1} & \textcolor{green}{1} & \textcolor{orange}{1} \\
 \textcolor{red}{1} & \textcolor{green}{1} & \textcolor{orange}{1} \\
\vspace{5mm}
 \textcolor{red}{1} \\
 \textcolor{green}{1} \\
 \textcolor{green}{1} & \textcolor{orange}{1} & \textcolor{blue}{1} & 1 \\
 \textcolor{green}{1} & \textcolor{orange}{1} & \textcolor{blue}{2} & 1 & 1 \\
 \textcolor{green}{1} & \textcolor{orange}{1} & \textcolor{blue}{1} & 1 \\
\vspace{5mm}
 \textcolor{green}{1} \\
 \textcolor{orange}{1} \\
 \textcolor{orange}{1} & \textcolor{blue}{1} & 1 & 1 & 1 \\
 \textcolor{orange}{1} & \textcolor{blue}{1} & 2 & 2 & 2 & 1 & 1 \\
 \textcolor{orange}{1} & \textcolor{blue}{1} & 2 & 2 & 2 & 1 & 1 \\
 \textcolor{orange}{1} & \textcolor{blue}{1} & 1 & 1 & 1 \\
\vspace{5mm}
 \textcolor{orange}{1} \\
 \textcolor{blue}{1} \\
 \textcolor{blue}{1} & 1 & 1 & 1 & 1 & 1 \\
 \textcolor{blue}{1} & 1 & 2 & 2 & 3 & 2 & 2 & 1 & 1 \\
 \textcolor{blue}{1} & 1 & 2 & 3 & 3 & 3 & 3 & 2 & 1 & 1 \\
 \textcolor{blue}{1} & 1 & 2 & 2 & 3 & 2 & 2 & 1 & 1 \\
 \textcolor{blue}{1} & 1 & 1 & 1 & 1 & 1 \\
\vspace{5mm}
 \textcolor{blue}{1} \\
 1 \\
 1 & 1 & 1 & 1 & 1 & 1 & 1 \\
 1 & 1 & 2 & 2 & 3 & 3 & 3 & 2 & 2 & 1 & 1 \\
 1 & 1 & 2 & 3 & 4 & 4 & 5 & 4 & 4 & 3 & 2 & 1 & 1 \\
 1 & 1 & 2 & 3 & 4 & 4 & 5 & 4 & 4 & 3 & 2 & 1 & 1 \\
 1 & 1 & 2 & 2 & 3 & 3 & 3 & 2 & 2 & 1 & 1 \\
 1 & 1 & 1 & 1 & 1 & 1 & 1 \\
\vspace{5mm}
 1 \\

\end{array}
$$

Each line in each grid represents some q-binomial coefficient. For example, the lines in the fifth gird correspond to the coefficients of the following Gaussian polynomials:
\begin{align*}
&\begin{bmatrix} 4 \\ 0 \end{bmatrix}=1 \\
&\begin{bmatrix} 4 \\ 1 \end{bmatrix}=1+q+q^2+q^3 \\
&\begin{bmatrix} 4 \\ 2 \end{bmatrix}=1+q+2q^2+q^3+q^4 \\
&\begin{bmatrix} 4 \\ 3 \end{bmatrix}=1+q+q^2+q^3 \\
&\begin{bmatrix} 4 \\ 4 \end{bmatrix}=1
\end{align*}

Now, it turns out that the sum of the numbers in red gives number of partitions of 4. Numbers in green add up to $p(5)$, numbers in orange to $p(6)$, numbers in blue to $p(7)$ and so on. Therefore, to each coefficient in each Gaussian polynomial there should correspond some partitions.

It is well known that the coefficient of $q^r$ in $\begin{bmatrix} n+m \\ n \end{bmatrix}$ represents the number of partitions of $r$ with $m$ or fewer parts each less than or equal to $n$. However, in our case this interpretation does not make any sense. But when we consider coefficient of $q^r$ in $\begin{bmatrix} n+m \\ n \end{bmatrix}$ to be equal to the number of partitions of $m+n+1+r$ with $n+1$ parts greatest part being equal to $m+1$ then everything falls in its place. In mathematical notation the previous sentence can be written as 

\begin{displaymath}
\begin{bmatrix} n+m \\ n \end{bmatrix}=\sum_{r\ge0}p(n+m+1+r\,|\,n+1\,\text{parts},\text{greatest part}=m+1)q^r
\end{displaymath}


For example, consider the following Gaussian polynomials

\begin{align*}
&\begin{bmatrix} 6 \\ 0 \end{bmatrix}=1 \\
&\begin{bmatrix} 6 \\ 0 \end{bmatrix}=1+q+q^2+q^3+q^4+q^5 \\
&\begin{bmatrix} 6 \\ 0 \end{bmatrix}=1+q+2q^2+2q^3+3q^4+2q^5+2q^6+q^7+q^8 \\
&\begin{bmatrix} 6 \\ 0 \end{bmatrix}=1+q+2q^2+3q^3+3q^4+3q^5+3q^6+2q^7+q^8+q^9 \\
&\begin{bmatrix} 6 \\ 0 \end{bmatrix}=1+q+2q^2+2q^3+3q^4+2q^5+2q^6+q^7+q^8 \\
&\begin{bmatrix} 6 \\ 0 \end{bmatrix}=1+q+q^2+q^3+q^4+q^5 \\
&\begin{bmatrix} 6 \\ 0 \end{bmatrix}=1 
\end{align*}

The table below shows partitions corresponding to these polynomials.

\begin{center}
\begin{tabular}{cp{1cm}p{0.8cm}p{0.8cm}p{0.8cm}p{0.8cm}p{0.8cm}p{0.8cm}p{0.8cm}p{0.8cm}p{0.8cm}}
\toprule

$\begin{bmatrix} n \\ m \end{bmatrix}$     	& $q^0$            	& $q^1$      	& $q^2$		& $q^3$		& $q^4$					&	$q^5$		& 	$q^6$		&	$q^7$ 	&	$q^8$	&	$q^9$	\\
\midrule
$\begin{bmatrix} 6 \\ 0 \end{bmatrix}$    	& 1111111         &	&	& 	& 	&	&	&	&	&\\[10pt]
$\begin{bmatrix} 6 \\ 1 \end{bmatrix}$    	& 111112           & 111122  	& 111222 	& 112222	& 122222 	&	222222	&	&	&	&\\[10pt]
$\begin{bmatrix} 6 \\ 2 \end{bmatrix}$    	& 11113             & 11123  	& 11133 11223 	& 11233 12223	& 11333  12233 22223	&	12333 22233	&	13333 22333	&	23333		&	33333		&\\[10pt]
$\begin{bmatrix} 6 \\ 3 \end{bmatrix}$    	& 1114               & 1124  	& 1134 1224 	& 1144 1234 2224	& 1244  1334 2234	&	1344 2244 2334	&	1444 2344 3334	&	2444 3344	&	3444		&	4444	\\[10pt]
$\begin{bmatrix} 6 \\ 4 \end{bmatrix}$   	& 115       	     & 125   	& 135 225		& 145 235	   	& 155 245 335			&	255 345		&	355 445	&	455	&	555	&		\\[10pt]
$\begin{bmatrix} 6 \\ 5 \end{bmatrix}$    	& 16           & 26  	& 36 	& 46	& 56 	&	66	&	&	&	&\\[10pt]
$\begin{bmatrix} 6 \\ 0 \end{bmatrix}$    	& 7         &	&	& 	& 	&	&	&	&	& \\
\bottomrule
\end{tabular}
\end{center}
\vspace{10mm} 

Looking at the table one may notice that center simmetric partitions are conjugate to each other. This fact can be generalized as follows: partitions corresponding to $q^r$ in $\begin{bmatrix} n \\ k \end{bmatrix}$ are conjugate to the partitions corresponding to $q^r$ in $\begin{bmatrix} n \\ n-k \end{bmatrix}$.

Remember that

\begin{displaymath}
\begin{bmatrix} n+m \\ n \end{bmatrix}=\sum_{r\ge0}p(r\,|\,\le m\,\text{parts},\text{each}\le n)q^r.
\end{displaymath}

At the same time

\begin{displaymath}
\begin{bmatrix} n+m \\ n \end{bmatrix}=\sum_{r\ge0}p(n+m+1+r\,|\,n+1\,\text{parts},\text{greatest part}=m+1)q^r.
\end{displaymath}

Putting these two facts together we get

\begin{displaymath}
p(n+m+1+r\,|\,n+1\,\text{parts},\text{greatest part}=m+1)=p(r\,|\,\le m\,\text{parts},\text{each}\le n)
\end{displaymath}

for $r=0,1,2,\ldots\,$.

It is well known that $\begin{bmatrix} n \\ m \end{bmatrix}=\begin{bmatrix} n \\ n-m \end{bmatrix}$. On the other hand, 

\begin{displaymath}
\begin{bmatrix} n \\ m \end{bmatrix}=\sum_{r\ge0}p(n+1+r\,|\,\text{into}\,n-m+1\,\text{parts, greatest}=m+1)q^r
\end{displaymath}

and

\begin{displaymath}
\begin{bmatrix} n \\ n-m \end{bmatrix}=\sum_{r\ge0}p(n+1+r\,|\,\text{into}\,m+1\,\text{parts, greatest}=n-m+1)q^r.
\end{displaymath}

Therefore, 

\begin{multline*}
p(n+1+r\,|\,\text{into}\,n-m+1\,\text{parts, greatest}=m+1)= \\
p(n+1+r\,|\,\text{into}\,m+1\,\text{parts, greatest}=n-m+1)
\end{multline*}






\section{Conclusion}

Partition function plays very important role in many areas of mathematics. For example:

\begin{enumerate}

\item 	$p(n)$ counts the number of conjugacy classes in the symmetric group $S_n$, which is at the same time the number of irreducible representations of $S_n$.

\item 	Number of distinct Abelian groups of order $p^n$ equals to $p(n)$.

\end{enumerate}


\begin{thebibliography}{9}
\bibitem{Sills} Yuriy Choliy and Andrew Sills, A Formula for the Partition Function that ''Counts"
\bibitem{Ramanujan} G. H. Hardy and S. Ramanujan, Asymptotic Formulae in Combinatory Analysis
\bibitem{Rademacher} Hans Rademacher, Topics in Analytic Number Theory, Springer-Verlag, 1973.
\bibitem{OnoBruinier} Ken Ono and Jan Hendrik Bruinier, Algebraic Formulas for the Coefficients of Half-Integral Weight Harmonic Weak Maass Forms
\bibitem{Andrews} George E. Andrews, The Theory of Partitions, Addison-Wesley, 1976.
\end{thebibliography}
\end{document}